\documentclass[a4paper,11pt]{amsart}

\usepackage{hyperref} 
\usepackage[backrefs]{amsrefs} 
\usepackage{microtype} 
\usepackage{verbatim} 
\usepackage{color}  
\usepackage{amsmath,amssymb} 
\usepackage{enumerate} 
\usepackage{mathtools} 
\usepackage{bm}

\title[A remark on torsion growth in homology]{A remark on torsion growth in homology\\ and volume of 3-manifolds}
\author{Holger Kammeyer}
\address{Karlsruhe Institute of Technology\\ Institute for Algebra and Geometry\\ Germany}
\email{holger.kammeyer@kit.edu}
\urladdr{www.math.kit.edu/iag7/~kammeyer/}


\newtheorem{theorem}{Theorem}
\newtheorem{corollary}{Corollary}
\newtheorem{proposition}{Proposition}

\newtheorem{conjecture}{Conjecture}
\newtheorem{question}{Question}

\newtheorem{introconj}{Conjecture}

\theoremstyle{definition}

\theoremstyle{remark}
\newtheorem{remark}{Remark}
\newtheorem{example}{Example}

\makeatletter
   \let\c@corollary=\c@theorem
   \let\c@proposition=\c@theorem
   \let\c@lemma=\c@theorem
   \let\c@definition=\c@theorem
   \let\c@remark=\c@theorem
   \let\c@example=\c@theorem
   \let\c@equation=\c@theorem
   \let\c@conjecture=\c@theorem
   \let\c@question=\c@theorem
\makeatother

\newcommand*{\MRref}[2]{ \href{http://www.ams.org/mathscinet-getitem?mr=#1}{MR \textbf{#1}}}
\newcommand*{\arXiv}[1]{ \href{http://www.arxiv.org/abs/#1}{arXiv:\textbf{#1}}}

\newcommand*{\Z}{\mathbb Z}
\newcommand*{\Q}{\mathbb Q}
\newcommand*{\R}{\mathbb R}
\newcommand*{\C}{\mathbb C}

\DeclareMathOperator{\tr}{tr}
\DeclareMathOperator{\rank}{rank}

\DeclareMathOperator{\Hom}{Hom}
\DeclareMathOperator{\epi}{Epi}
\DeclareMathOperator{\vol}{vol}

\DeclarePairedDelimiter\floor{\lfloor}{\rfloor}

\newcounter{commentcounter}

\usepackage{ifthen,srcltx}
\newcommand{\showcomments}{yes}

\newsavebox{\commentbox}
\newenvironment{com}%
{\ifthenelse{\equal{\showcomments}{yes}}%
{\footnotemark
        \begin{lrbox}{\commentbox}
        \begin{minipage}[t]{1.25in}\raggedright\sffamily\tiny
        \footnotemark[\arabic{footnote}]}
{\begin{lrbox}{\commentbox}}}
{\ifthenelse{\equal{\showcomments}{yes}}
{\end{minipage}\end{lrbox}\marginpar{\usebox{\commentbox}}}
{\end{lrbox}}}

\newcommand{\ignore}[1]{} 

\begin{document}

\begin{abstract}  
We show that L{\"u}ck's conjecture on torsion growth in homology implies that two 3-manifolds have equal volume if the fundamental groups have the same set of finite quotients.
\end{abstract} 

\maketitle

The purpose of this note is to relate two well-known open problems which both deal with a residually finite fundamental group \(\Gamma\) of an odd-dimensional aspherical manifold.  The first one \cite{Lueck:GAFA}*{Conjecture~1.12(2)} predicts that the \emph{\(\ell^2\)-torsion} \(\rho^{(2)}(\Gamma)\) determines the exponential rate at which torsion in middle-degree homology grows along a chain of finite index normal subgroups.

\begin{introconj} \label{conj:torsionapprox}
    Let \(M\) be an aspherical closed manifold of dimension \(2d + 1\).  Suppose that \(\Gamma = \pi_1 M\) is residually finite and let \(\Gamma = \Gamma_0 \ge \Gamma_1 \ge \cdots\) be any chain of finite index normal subgroups of \(\Gamma\) with \(\bigcap_{n=0}^\infty \Gamma_n = \{1\}\).  Then
    \[ \lim_{n \rightarrow \infty} \frac{\log |H_d(\Gamma_n)_{\textup{tors}}|}{[\Gamma : \Gamma_n]} = (-1)^d \rho^{(2)}(\Gamma).\]
\end{introconj}

The term \(|H_d(\Gamma_n)_{\textup{tors}}|\) denotes the order of the torsion subgroup of \(H_d(\Gamma_n)\).  The \(\ell^2\)-torsion \(\rho^{(2)}(\Gamma)\) is the \(\ell^2\)-counterpart to Reidemeister torsion as surveyed in \cite{Lueck:L2Invariants} and \cite{Kammeyer:IntroL2}.  The second conjecture says that volume of 3-manifolds can be recovered from the finite quotients of the fundamental group.

\begin{introconj} \label{conj:volumeprofinite}
  Let \(\Gamma\) and \(\Lambda\) be infinite fundamental groups of connected, closed, orientable, irreducible 3-manifolds and suppose that \(\widehat{\Gamma} \cong \widehat{\Lambda}\).  Then
  \[ \vol(\Gamma) = \vol(\Lambda). \]
\end{introconj}

Here the \emph{profinite completion} \(\widehat{\Gamma}\) of \(\Gamma\) is the projective limit over all finite quotients of \(\Gamma\).  Two groups have isomorphic profinite completions if and only if they have the same set of finite quotients \cite{Ribes-Zalesskii:ProfiniteGroups}*{Corollary~3.2.8}.  If \(\Gamma = \pi_1 M\) for a 3-manifold~\(M\) with the stated properties, then \emph{Thurston geometrization} applies to \(M\): there is a minimal choice of finitely many disjointly embedded incompressible tori in \(M\), unique up to isotopy, which cut \(M\) into pieces such that each piece carries one out of eight geometries.  The sum of the volumes of the hyperbolic pieces gives the well-defined quantity \(\vol(\Gamma)\).  Conjecture~\ref{conj:volumeprofinite} is often stated as a question \cite{Boileau-Friedl:ProfiniteCompletions}*{Question~3.18}.  But we dare to promote it to a conjecture in view of the following result.

\begin{theorem} \label{thm:mainthm}
  Conjecture~\ref{conj:torsionapprox} implies Conjecture~\ref{conj:volumeprofinite}.
\end{theorem}

The theorem seems to be folklore among the experts in the field but I could not find a proof in the literature so that this note is meant as a service to the community.

The contrapositive of Theorem~\ref{thm:mainthm} says that constructing two profinitely isomorphic 3-manifold groups with differing covolume would disprove Conjecture~\ref{conj:torsionapprox}.  Funar~\cite{Funar:TorusBundles} and Hempel~\cite{Hempel:Some3ManifoldGroups} constructed examples of closed 3-manifolds with non-isomorphic but profinitely isomorphic fundamental groups.  These examples carry \(\textit{Sol}\) and \(\mathbb{H}^2 \times \R\) geometry, respectively, and thus all have zero volume by definition.  Wilkes~\cite{Wilkes:ProfiniteRigidity} showed that Hempel's examples are the only ones among Seifert-fiber spaces.  It seems to be open whether there exist such examples with \(\mathbb{H}^3\)-geometry.  As a first step in the negative direction, Bridson and Reid \cite{Bridson-Reid:ProfiniteRigidity} showed that the figure eight knot group is determined among 3-manifold groups by the profinite completion.  

The paper at hand is divided into two sections.  Section~\ref{section:proof} presents the proof of Theorem~\ref{thm:mainthm}.  As a complement, Section~\ref{section:relatedconjectures} discusses how the related \emph{asymptotic volume conjecture} and the \emph{Bergeron--Venkatesh conjecture} fit into the picture.

I wish to thank S.\,Kionke and J.\,Raimbault for helpful discussions during the junior trimester program ``Topology'' at the Hausdorff Research Institute for Mathematics in Bonn.

\section{Proof of Theorem~\ref{thm:mainthm}} \label{section:proof}

For the moment, let \(\Gamma\) and \(\Lambda\) be any two finitely generated, residually finite groups.  To prepare the proof of Theorem~\ref{thm:mainthm}, we collect a couple of propositions from the survey article \cite{Reid:ProfiniteProperties} and include more detailed proofs for the sake of a self-contained treatment.  We first recall that the open subgroups of \(\widehat{\Gamma}\) are precisely the subgroups of finite index.  One direction is easy: \(\widehat{\Gamma}\) is compact and the cosets of an open subgroup form a disjoint open cover.  The converse is a deep theorem due to Nikolov and Segal \cite{Nikolov-Segal:ProfiniteGroups} that crucially relies on the assumption that \(\Gamma\) is finitely generated.  The proof moreover invokes the classification of finite simple groups.

  The assumption that \(\Gamma\) is residually finite says precisely that the canonical map \(\Gamma \rightarrow \widehat{\Gamma}\) is an embedding.  If \(Q\) is a finite group, the universal property of \(\widehat{\Gamma}\) says that the restriction map \(\Hom(\widehat{\Gamma}, Q) \rightarrow \Hom(\Gamma, Q)\) is a surjection.  By the above, the kernel of any homomorphism \(\widehat{\Gamma} \xrightarrow{\varphi} Q\) is open which implies that \(\varphi\) is continuous and is thus determined by the values on the dense subset \(\Gamma \subset \widehat{\Gamma}\).  Thus \(\Hom(\widehat{\Gamma}, Q) \rightarrow \Hom(\Gamma, Q)\) is in fact a bijection which clearly restricts to a bijection \(\epi(\widehat{\Gamma}, Q) \rightarrow \epi(\Gamma, Q)\) of surjective homomorphisms.  This has the following consequence.

\begin{proposition} \label{prop:h1epimorphism}
  If \(\Lambda\) embeds densely into \(\widehat{\Gamma}\), then there is an epimorphism \(H_1(\Lambda) \rightarrow H_1(\Gamma)\).
\end{proposition}

\begin{proof}
Let \(p\) be a prime number which does not divide the group order \(|H_1(\Lambda)_{\textup{tors}}|\) and let us set \(r = \dim_\Q H_1(\Gamma; \Q)\).  It is apparent that we have an epimorphism \(\Gamma \rightarrow (\Z/p\Z)^r \oplus H_1(\Gamma)_{\textup{tors}}\).  By the above remarks, this epimorphism extends uniquely to an epimorphism \(\widehat{\Gamma} \rightarrow (\Z/p\Z)^r \oplus H_1(\Gamma)_{\textup{tors}}\).  Since \(\Lambda\) embeds densely into \(\widehat{\Gamma}\), the latter map restricts to an epimorphism \(\Lambda \rightarrow (\Z / p\Z)^r \oplus H_1(\Gamma)_{\textup{tors}}\).  This epimorphism must lift to an epimorphism \(\Lambda \rightarrow \Z^r \oplus H_1(\Gamma)_{\textup{tors}} \cong H_1(\Gamma)\) because \(p\) is coprime to \(|H_1(\Lambda)_{\textup{tors}}|\).  Of course this last epimorphism factors through the abelianization \(H_1(\Lambda)\).
\end{proof}

\begin{corollary} \label{cor:h1profinite}
  The abelianization is a profinite invariant: if \(\widehat{\Gamma} \cong \widehat{\Lambda}\), then \(H_1(\Gamma) \cong H_1(\Lambda)\).
\end{corollary}

\begin{proof}
  Since we have surjections in both directions the groups \(H_1(\Gamma)\) and \(H_1(\Lambda)\) have the same free abelian rank.  Thus either surjection restricts to an isomorphism of the free parts and thus induces a surjection of the finite torsion quotients---which then must be a bijection.
\end{proof}

Let us now endow \(\Gamma\) with the subspace topology of \(\widehat{\Gamma}\), called the \emph{profinite topology} of \(\Gamma\).  For the open subgroups of \(\Gamma\) we have the same situation as we observed for \(\widehat{\Gamma}\).

\begin{proposition} \label{prop:profinitetopology}
  A subgroup \(H \le \Gamma\) is open in the profinite topology if and only if \(H\) has finite index in \(\Gamma\).
\end{proposition}

\begin{proof}
  Recall that \(\widehat{\Gamma}\) carries the coarsest topology under which the projections \(\widehat{\Gamma} \rightarrow \Gamma / \Gamma_i\) for finite index normal subgroups \(\Gamma_i \trianglelefteq \Gamma\) are continuous.  Since the compositions \(\Gamma \rightarrow \widehat{\Gamma} \rightarrow \Gamma / \Gamma_i\) are the canonical projections, it follows that a subbase for the subspace topology of \(\Gamma \subset \widehat{\Gamma}\) is given by the cosets of finite index normal subgroups of \(\Gamma\).

  If \(H\) has finite index in \(\Gamma\), then so does the \emph{normal core} \(N = \bigcap_{g \in \Gamma} gHg^{-1}\) because \(N\) is precisely the kernel of the permutation representation of \(\Gamma\) on the homogeneous set \(\Gamma / H\) defined by left translation.  Thus \(H = \bigcup_{h \in H} hN\) is open.  Conversely, let \(H \le \Gamma\) be open. Then \(H\) is a union of finite intersections of finite index normal subgroups of \(\Gamma\).  In particular \(H\) contains a finite index subgroup, whence has finite index itself.
\end{proof}

  \begin{proposition} \label{prop:11correspondence}
    Taking closure \(H \mapsto \overline{H}\) in \(\widehat{\Gamma}\) defines a 1-1--correspondence from the open (or finite index) subgroups of \(\Gamma\) to the open (or finite index) subgroups of \(\widehat{\Gamma}\).  The inverse is given by intersection \(H \mapsto H \cap \Gamma\) with \(\Gamma\).  This correspondence preserves the index, sends a normal subgroup \(N \trianglelefteq \Gamma\) to a normal subgroup \(\overline{N} \trianglelefteq \widehat{\Gamma}\), and in the latter case we have \(\widehat{\Gamma} / \overline{N} \cong \Gamma / N\).
  \end{proposition}

  The proof is given in \cite{Ribes-Zalesskii:ProfiniteGroups}*{Prop.\,3.2.2, p.\,84}.  Here is an easy consequence.
  \begin{corollary}
    For \(H_1, H_2 \le \Gamma\) of finite index we have \(\overline{H_1 \cap H_2} = \overline{H_1} \cap \overline{H_2}\).
  \end{corollary}

  \begin{proof}
    By the proposition \(\overline{H_1} \cap \overline{H_2}\) has finite index in \(\widehat{\Gamma}\) and we get
    \[ (\overline{H_1} \cap \overline{H_2}) \cap \Gamma = (\overline{H_1} \cap \Gamma) \cap (\overline{H_2} \cap \Gamma) = H_1 \cap H_2. \]
    Applying the proposition again yields \(\overline{H_1} \cap \overline{H_2} = \overline{H_1 \cap H_2}\).
  \end{proof}
  
  Note that for a finitely generated, residually finite group \(\Gamma\) there is a canonical choice of a chain
  \[ \Gamma = M_1 \ge M_2 \ge M_3 \ge \cdots \]
  of finite index normal subgroups \(M_n \trianglelefteq \Gamma\) satisfying \(\bigcap_{n=1}^\infty M_n = \{ 1 \}\).  Simply define \(M_n\) to be the intersection of the (finitely many!) normal subgroups of index at most \(n\).  By the last two results, \(\overline{M_n}\) is the intersection of all normal subgroups of \(\widehat{\Gamma}\) with index at most \(n\).

  \begin{proposition} \label{prop:intersectiontrivial}
    The intersection \(\bigcap_{n = 1}^\infty \overline{M_n}\) is trivial.
  \end{proposition}

  \begin{proof}
    Let \(g \in \bigcap_{n = 1}^\infty \overline{M_n} \subset \widehat{\Gamma}\).  Since \(\Gamma\) is finitely generated, it has only countably many subgroups of finite index.  Therefore the description of the topology of \(\widehat{\Gamma}\) given above shows that \(\widehat{\Gamma}\) is second and thus first countable.  Hence we can pick a sequence \((g_i)\) from the dense subset \(\Gamma \subset \widehat{\Gamma}\) with \(\lim_{i \rightarrow \infty} g_i = g\).  Let \(p_n \colon \Gamma \rightarrow \Gamma / M_n\) and \(\widehat{p}_n \colon \widehat{\Gamma} \rightarrow \widehat{\Gamma} / \overline{M_n}\) denote the canonical projections.  Since \(\widehat{p}_n\) is continuous, we have
    \[ \overline{M_n} = \widehat{p}_n(g) = \lim_{i \rightarrow \infty} \widehat{p}_n(g_i) \]
    and hence \(\lim_{i \rightarrow \infty} p_n(g_i) = M_n \in \Gamma / M_n\) because \(\widehat{\Gamma} / \overline{M_n} \cong \Gamma / M_n\) by Proposition~\ref{prop:11correspondence}.  As \(\Gamma / M_n\) is discrete, the sequence \(p_n(g_i)\) is eventually constant.  This means that for all \(n \ge 1\) there is \(N \ge 1\) such that for all \(i \ge N\) we have \(p_n(g_i) = M_n\), or equivalently \(g_i \in M_n\).  But the open sets \(M_n\) form a neighborhood basis of \(1 \in \Gamma\) as follows from the description of the profinite topology of \(\Gamma\) given in the proof of Proposition~\ref{prop:profinitetopology}.  So the last statement gives \(\lim_{i \rightarrow \infty} g_i = 1\).  Since \(\widehat{\Gamma}\) (and hence \(\Gamma\)) is Hausdorff, we conclude \(g = 1\).
  \end{proof}

  It follows that \(\widehat{\Gamma}\) is residually finite as an abstract group.  Before we give the proof of Theorem~\ref{thm:mainthm}, we put down one more observation.  If \(H \le \Gamma\) is any subgroup, then the closure \(\overline{H}\) in \(\widehat{\Gamma}\) is a profinite group so that the universal property of \(\widehat{H}\) gives a canonical homomorphism \(\eta \colon \widehat{H} \rightarrow \overline{H}\) which restricts to the identity on \(H\).  This is always an epimorphism because the image is dense, as it contains \(H\), and closed because it is compact and \(\overline{H}\) is Hausdorff.  However, in general we cannot expect that \(\eta\) is injective, not even if \(H\) is finitely generated.  Nevertheless:

  \begin{proposition} \label{proposition:closureequalscompletion}
    If \(H \le \Gamma\) has finite index, then the canonical map \(\eta \colon \widehat{H} \rightarrow \overline{H}\) is an isomorphism.
  \end{proposition}

  \begin{proof}
    Let \(h \in \ker \eta\).  The group \(H\) is finitely generated because it is a finite index subgroup of \(\Gamma\).  As above we conclude that \(\widehat{H}\) is second and hence first countable.  Since \(H\) lies densely in \(\widehat{H}\), we can thus pick a sequence of elements \(h_i \in H\) such that \(\lim_{i \rightarrow \infty} h_i = h\).  By continuity of \(\eta\), we obtain \(\lim_{i \rightarrow \infty} \eta(h_i) = \eta(h) = 1\) and thus \(\lim_{i \rightarrow \infty} h_i = 1\) in the topology of \(\overline{H}\).  A neighborhood basis of \(1 \in \overline{H}\) is given by the sets \(\overline{M_n} \cap \overline{H}\) where \(\overline{M_n}\) are the finite index normal subgroups of \(\widehat{\Gamma}\) from above.  It follows that for all \(n \ge 1\) there exists \(N \ge 1\) such that for all \(i \ge N\) we have \(h_i \in M_n \cap H\).  Since \(H\) has finite index in \(\Gamma\), it follows that any finite index normal subgroup \(K \trianglelefteq H\) has also finite index as a subgroup of \(\Gamma\).  Thus there exists \(n \ge 1\) such that \(M_n\) lies in the normal core of \(K\) as a subgroup of \(\Gamma\).  Hence for all \(K \trianglelefteq H\) of finite index there exists \(N \ge 1\) such that for all \(i \ge N\) we have \(h_i \in K\).  But the finite index normal subgroups \(K \trianglelefteq H\) form a neighborhood basis of \(1 \in H\) in the profinite topology of \(H\).  Hence we have \(\lim_{i \rightarrow \infty} h_i = 1\) in the topology of \(\widehat{H}\).  Since \(\widehat{H}\) is Hausdorff, we conclude \(h = 1\).
  \end{proof}

  \begin{proof}[Proof of Theorem~\ref{thm:mainthm}.]
  Note that \(\Gamma\) and \(\Lambda\) are finitely generated and residually finite, as a consequence of geometrization \cite{Hempel:ResidualFiniteness}.  We fix an isomorphism \(\widehat{\Gamma} \cong \widehat{\Lambda}\).  Again, let \(M_n \le \Gamma\) be the intersection of all normal subgroups of \(\Gamma\) of index at most \(n\).  By Proposition~\ref{prop:11correspondence} it follows that \(L_n = \Lambda \cap \overline{M_n}\) is the intersection of all normal subgroups of \(\Lambda\) of index at most \(n\) and \([\Gamma : M_n] = [\Lambda : L_n]\).  By Proposition~\ref{prop:intersectiontrivial} we have \(\bigcap_n \overline{M_n} = \{1\}\) so that \(\bigcap_n L_n = \{1\}\).  From Proposition~\ref{proposition:closureequalscompletion} we get \(\widehat{M_n} \cong \widehat{L_n}\) so that Corollary~\ref{cor:h1profinite} implies \(|H_1(M_n)_{\textup{tors}}| = |H_1(L_n)_{\textup{tors}}|\).  A theorem of L\"uck and Schick~\cite{Lueck-Schick:Hyperbolic}*{Theorem~0.7} conjectured in Lott and L\"uck~\cite{Lott-Lueck:3manifolds}*{Conjecture~7.7} shows that \(\rho^{(2)}(\Gamma) = -\vol(\Gamma) / 6 \pi\) and similarly for \(\Lambda\), see also \cite{Lueck:L2Invariants}*{Theorem~4.3, p.\,216}.  If Conjecture~\ref{conj:torsionapprox} holds true, this implies
  \[ \vol(\Gamma) = 6 \pi \lim_{n \rightarrow \infty} \frac{\log |H_1(M_n)_{\textup{tors}}|}{[\Gamma : M_n]} = 6 \pi \lim_{n \rightarrow \infty} \frac{\log |H_1(L_n)_{\textup{tors}}|}{[\Lambda : L_n]} = \vol (\Lambda). \qedhere \]
\end{proof}

\section{Related conjectures} \label{section:relatedconjectures}

  
One can find companion conjectures to Conjecture~\ref{conj:torsionapprox} in the literature which likewise predict an exponential rate of torsion growth in homology proportional to volume.  However, these conjectures restrict the aspherical manifolds under consideration in one way or another.  Specifically dealing with 3-manifolds is L\^e's \emph{asymptotic volume conjecture}. 

\begin{introconj} \label{conj:asymptoticvolume}
  Let \(\Gamma\) be the fundamental group of a connected, orientable, irreducible, compact 3-manifold whose boundary is either empty or a collection of tori.  Then
  \[ \limsup_{\Gamma_n \rightarrow \{1\}} \frac{\log |H_1(\Gamma_n)_{\textup{tors}}|}{[\Gamma : \Gamma_n]} = \frac{\vol(\Gamma)}{6 \pi}. \]
\end{introconj}

The conjecture appears in \cite{Le:GrowthOfHomology}*{Conjecture~1\,(a)}.  The volume \(\vol(\Gamma)\) is defined by a geometric decomposition as before which also exists for toroidal boundary.  The \(\limsup\) on the left hand side is defined as the lowest upper bound of all \(\limsup\)s along sequences \((\Gamma_n)\) of (not necessarily nested!) finite index normal subgroups of \(\Gamma\) with \(\limsup_n \Gamma_n = \{ 1 \}\).  Recall that by definition
\[ \limsup\nolimits_n \Gamma_n = \bigcap_{N \ge 0} \bigcup_{n \ge N} \Gamma_n \]
so that the condition \(\limsup_n \Gamma_n = \{ 1 \}\) is actually equivalent to requiring
\[ \lim_{n \rightarrow \infty} \tr_{\C[\Gamma / \Gamma_n]}(g\Gamma_n) = \tr_{\C[\Gamma]}(g) = \begin{cases} 1 \text{ if } g = e,\\
  0 \text{ otherwise,} \end{cases} \]
for all \(g \in \Gamma\) where the traces are the usual traces of group algebras given by the unit coefficient.

\begin{question} \label{question:asymptoticvolumeimpliesvolumeprofinite}
Does Conjecture~\ref{conj:asymptoticvolume} imply Conjecture~\ref{conj:volumeprofinite}?
\end{question}

The proof of Theorem~\ref{thm:mainthm} does not immediately carry over to Question~\ref{question:asymptoticvolumeimpliesvolumeprofinite} as \(\limsup_n \Gamma_n = \{ 1 \}\) for some sequence \((\Gamma_n)\) does not imply \(\limsup \Lambda_n = \{ 1 \}\) for the groups \(\Lambda_n = \Lambda \cap \overline{\Gamma_n}\).  Here is an example.

\begin{example}
  Let \(\Gamma = \Z \times \Z\) with (nested) chain of subgroups \(\Gamma_n = 2^n \Z \times 3^n \Z\).  Clearly, we have \(\widehat{\Gamma} = \widehat{\Z} \times \widehat{\Z}\).  From the description \(\widehat{\Z} \cong \prod_p \Z_p\) it is apparent that \([\widehat{\Z} : N \widehat{\Z}] = N\) for any \(N \ge 1\).  Since \(N\Z\) is the only subgroup of index \(N\) in~\(\Z\), Proposition~\ref{prop:11correspondence} implies that \(\overline{N\Z} = N \widehat{\Z}\).  Thus we have \(\overline{\Gamma_n} = 2^n \widehat{\Z} \times 3^n \widehat{\Z}\).  It follows that
  \[ \bigcap_{n = 1}^\infty \overline{\Gamma_n} \cong \left( \{0\} \times \prod_{p > 2} \Z_p \right) \times \left(\Z_2 \times \{0\} \times \prod_{q > 3} \Z_q \right) \le \widehat{\Z} \times \widehat{\Z}. \]
  So if we let \(\Lambda \le \widehat{\Gamma}\) be the subgroup generated by the two elements
  \[ ((0,1,1,\ldots), (1,0,0,0,\ldots)) \quad \text{and} \quad ((1,0,0,0\ldots),(0,1,1,1,\ldots))\]
  in \(\prod_p \Z_p \times \prod_p \Z_p \cong \widehat{\Z} \times \widehat{\Z}\), then clearly \(\Lambda \cong \Z \times \Z\) is dense in \(\widehat{\Gamma}\) so that the canonical map \(\widehat{\Lambda} \rightarrow \overline{\Lambda} = \widehat{\Gamma}\) is a surjective homomorphism of isomorphic finitely generated profinite groups.  Hence it must be an isomorphism~\cite{Ribes-Zalesskii:ProfiniteGroups}*{Proposition~2.5.2, p.\,46}.  However, we have \(\bigcap_{n = 1}^\infty \Lambda_n \neq \{0\}\) even though \(\bigcap_{n=1}^\infty \Gamma_n = \{0\}\).
\end{example}

We remark that L{\^e} has proven the inequality ``\(\le\)'' of Conjecture~\ref{conj:asymptoticvolume}, even if the subgroups are not required to be normal.  Another conjecture, which leaves both the realm of 3-manifolds and of normal subgroups, is due to Bergeron and Venkatesh~\cite{Bergeron-Venkatesh:AsymptoticGrowth}*{Conjecture~1.3}.  It does however assume a somewhat rigorous arithmetic setting.  This is what we want to present next.

Let \(\mathbf{G}\) be a semisimple algebraic group, defined and anisotropic over \(\Q\).  Let \(\Gamma \le \mathbf{G}(\Q)\) be a \emph{congruence subgroup}.  This means that for some (and then for any) \(\Q\)-embedding \(\rho \colon \mathbf{G} \rightarrow \textup{GL}_n\) there is \(k \ge 1\) such that the group \(\rho(\Gamma)\) contains the kernel of \(\rho(\mathbf{G}) \cap \textup{GL}_n(\Z) \rightarrow \textup{GL}_n(\Z/k\Z)\) as a subgroup of finite index.  Fix an algebraic representation of \(\mathbf{G}\) on a finite-dimensional \(\Q\)-vector space \(W\) and let \(M \subset W\) be a \(\Gamma\)-invariant \(\Z\)-lattice, which always exists according to \cite{Platonov-Rapinchuk:AlgebraicGroups}*{Remark, p.\,173}.  Let \(\Gamma = \Gamma_0 \ge \Gamma_1 \ge \cdots\) be a chain of congruence subgroups with \(\bigcap_n \Gamma_n = \{1\}\).  For a maximal compact subgroup \(K\) of \(G = \mathbf{G}(\R)\), we denote by \(X = G / K\) the \emph{symmetric space} associated with \(\mathbf{G}\).  Let \(\mathfrak{g}\) and \(\mathfrak{k}\) be the Lie algebras of \(G\) and \(K\) and let \(\delta(G) = \rank_{\C} \,\mathfrak{g} \otimes \C - \rank_{\C} \,\mathfrak{k} \otimes \C\) be the \emph{deficiency} of \(G\), sometimes also known as the \emph{fundamental rank} \(\delta(X)\) of \(X\).

\begin{introconj} \label{conj:bergeronvenkatesh}
  For each \(d \ge 1\) there is a constant \(c_{G, M, d} \ge 0\) such that
    \[ \lim_{n \rightarrow \infty} \frac{\log |H_d(\Gamma_n; M)_\textup{tors}|}{[\Gamma : \Gamma_n]} = c_{G, M, d} \vol(\Gamma) \]
and \(c_{G,M,d} > 0\) if and only if \(\delta(G) = 1\) and \(\dim X = 2d + 1\).
\end{introconj}

In this case the volume \(\vol(\Gamma)\) is the volume of the closed locally symmetric space \(\Gamma \backslash X\) which is defined by means of a Haar measure on \(G\) and as such only unique up to scaling.  But any rescaling of this measure would also rescale the constant \(c_{G, M, d}\) by the reciprocal value so that the product is well-defined.  To make sure that \(c_{G, M, d}\) really only depends on \(G\), \(M\), and \(d\), we agree upon the following normalization of the Haar measure.  The Killing form on \(\mathfrak{g}\) restricts to a positive definite form on the subspace \(\mathfrak{p}\) in the orthogonal Cartan decomposition \(\mathfrak{g} = \mathfrak{k} \oplus \mathfrak{p}\).  Identifying \(\mathfrak{p}\) with the tangent space \(T_K X\), we obtain a \(G\)-invariant metric on \(X\) by translation.  We require that the volume of \(\Gamma \backslash X\) determined by Haar measure be equal to the volume of \(\Gamma \backslash X\) as Riemannian manifold.

To relate Conjecture~\ref{conj:bergeronvenkatesh} to Conjecture~\ref{conj:volumeprofinite}, we need to restrict our attention to \emph{arithmetic hyperbolic 3-manifolds}.  These are quotients of hyperbolic 3-space~\(\mathbb{H}^3\) by \emph{arithmetic Kleinian groups}.  A \emph{Kleinian group} is a discrete subgroup \(\Gamma \le \textup{PSL}(2, \C) \cong \textup{Isom}^+(\mathbb{H}^3)\) such that \(\vol(\Gamma) = \vol (\Gamma \backslash \mathbb{H}^3) < \infty\).  A Kleinian group \(\Gamma \le \textup{PSL}(2, \C)\) is called \emph{arithmetic} if there exists a semisimple linear algebraic \(\Q\)-group \(\mathbf{H} \le \textup{GL}_n\) and an epimorphism of Lie groups \(\phi \colon \mathbf{H}(\R)^0 \rightarrow \textup{PSL}(2, \C)\) with compact kernel such that \(\Gamma\) is commensurable with \(\phi(\mathbf{H}(\Z) \cap \mathbf{H}(\R)^0)\).  Here \(\mathbf{H}(\R)^0\) denotes the unit component and two subgroups of a third group are called commensurable if their intersection has finite index in both subgroups.  Note that we consider \(\textup{PSL}(2, \C)\) as a real Lie group so that the complexified lie algebra is \(\mathfrak{sl}(2, \C) \oplus \mathfrak{sl}(2, \C)\) and hence \(\delta(\textup{PSL}(2, \C)) = 1\).  There is an alternative and equivalent approach to the definition of arithmetic Kleinian groups via orders in quaternion algebras over number fields~\cite{Maclachlan-Reid:Arithmetic}.

\begin{question}
  Let \(\Gamma\) and \(\Lambda\) be arithmetic Kleinian groups such that \(\widehat{\Gamma} = \widehat{\Lambda}\).  Suppose Conjecture~\ref{conj:bergeronvenkatesh} holds true.  Can we conclude that \(\vol(\Gamma) = \vol(\Lambda)\)?
  \end{question}

Again, various problems arise when trying to adapt the proof of Theorem~\ref{thm:mainthm} to settle this question in the affirmative.  To be more concrete, a direct translation fails for the following reason.  Let \(M_n\) be the intersection of all normal subgroups of index at most \(n\) in the arithmetic group \(\mathbf{H}(\Z)\) corresponding to \(\Gamma\) as above.  Then \(M_n\) will not consist of congruence subgroups.  In fact, \(\mathbf{H}(\Z)\) has the \emph{congruence subgroup property} if and only if all the groups \(M_n\) are congruence subgroups.  But the congruence subgroup property is well known to fail for all arithmetic Kleinian groups~\cite{Lubotzky:GroupPresentations}.  Instead, one could try to start with a chain of congruence subgroups \(\Gamma_n\) of \(\Gamma\) but then it seems unclear if or under what circumstances the chain \(\Lambda_n = \overline{\Gamma_n} \cap \Lambda\) consists of congruence subgroups in \(\Lambda\).

We remark that for the trivial coefficient system \(\Z \subset \Q\), Conjecture~\ref{conj:bergeronvenkatesh} is wide open.  However, in our relevant case of \(\delta(G) = 1\), Bergeron and Venkatesh construct \emph{strongly acyclic} coefficient modules \(M\) with the property that the spectrum of the Laplacian acting on \(M \otimes_\Z \C\)-valued \(p\)-forms on \(\Gamma_n \backslash X\) is bounded away from zero for all \(p\) and \(n\).  In the special case \(G = \textup{SL}(2, \C)\), they show that Conjecture~\ref{conj:bergeronvenkatesh} holds true for any strongly acyclic \(M\).

\end{document}